\documentclass[12pt]{amsart}

\usepackage{latexsym,amsmath,amsfonts,amscd,amsthm,upref,graphics,amssymb}

\usepackage{mathrsfs}
\usepackage[all]{xy}
\usepackage{amsmath}
\usepackage{amssymb}
\usepackage{latexsym}
\usepackage{enumerate}
\usepackage{epsfig}
\usepackage{bar}
\usepackage{epic}
\usepackage{multicol}
\usepackage{graphicx}
\usepackage{epsfig}

\usepackage{amsfonts}

\usepackage{graphicx}
\usepackage{wrapfig}
\usepackage{color}
\usepackage{graphicx}
\usepackage{epsfig}

\usepackage{ragged2e}
\usepackage{lipsum}
\usepackage{epstopdf}

\usepackage{epstopdf}

\theoremstyle{plain}
\newtheorem{theorem}{Theorem}[section] 
\newtheorem{lemma}[theorem]{Lemma}     

\newtheorem{proposition}[theorem]{Proposition}
\newtheorem{remark}[theorem]{Remark}

\newtheorem*{theorem*}{Theorem}
\newtheorem*{definition*}{Definition}

\begin{document}
	\title{The flow of weights of some factors arising as fixed point algebras}
\author{Radu B. Munteanu}
\address[R. B. Munteanu]{Department of Mathematics, University of Bucharest, 14 Academiei St., 010014 and   
	Simion Stoilow Institute of Mathematics of the  Romanian Academy, 21 Calea Grivitei Street,  010702, Bucharest,  Romania}
\email{radu-bogdan.munteanu@g.unibuc.ro}
\renewcommand{\thefootnote}{}


\renewcommand{\thefootnote}{\arabic{footnote}}
\setcounter{footnote}{0}
\begin{abstract}
		In this paper we study the associated flow of some factors arising as fixed point algebras under product type actions on ITPFI factors. We compute their associated flow and show that under certain conditions these flows are approximately transitive and consequently the corresponding fixed point factors are ITPFI. 
	\end{abstract}
\maketitle
\section{Introduction}

A well known result of Krieger \cite{K} says that the flow of weights considered as a map between the isomorphism classes of approximately finite dimensional (AFD) factors of type III$_0$ to the conjugacy classes of properly ergodic flows is one to one and onto. We recall that with any countable equivalence relation $\mathcal{R}$ we can associate a von Neumann algebra  $M(\mathcal{R})$, via Feldman-Moore construction \cite{FM1, FM2}, and there is a bijective correpondence between the isomorphism classes of AFD factors and the orbit equivalence classes of hyperfinite equivalence relations. If $\mathcal{R}$ is a hyperfinite equivalence relation of type III then the flow of weights of  $M(\mathcal{R})$ is (up to conjugacy) the associated flow of $\mathcal{R}$. 

Approximate transitivity (abbreviated AT) is a property of ergodic actions introduced by Connes and Woods in \cite{CW}, which characterize the Araki-Woods (or ITPFI) factors among the AFD factors. They prove that an AFD factor of type III$_0$ is ITPFI if and only if its flow of weights is AT. Equivalently this result says that a group of ergodic automorphism is orbit equivalent to a product type odomoter if and only if its associated flow is AT. A measure theoretical proof of this result was given in \cite{H}. 

We recall that an action $\alpha$ of a Borel group $G$ on a measure space $(X,\nu)$ is approximately transitive if given $n<\infty$, functions $f_1,f_2, \ldots, f_n \in L^{1}_{+}(X,\nu)$ and $\varepsilon> 0,$
there exists a function $f\in L^{1}_{+}(X,\nu)$, elements $ g_1,\dots,g_m \ \in\ G$ for some
$m<\infty,$ and $\lambda_{j,k}\geq 0$, for $k=1,2,\dots,m$ and $j=1,2,\ldots,n$ such that
$$ \left\| f_j\ -\ \sum_{k=1}^m \lambda_{j,k}f\circ g_k\frac{d\nu g_k}{d \nu} \right\| \leq\varepsilon,\quad
j=1,\dots,n.$$
If $G=\mathbb{Z}$ and $\alpha$ is AT, then we say that $T=\alpha(1)$ is AT.  


In \cite{GH}, Giordano and Handelman came up with a different approach for studying dynamical systems. They used matrix valued random walks and their Poisson boundary to deal with measure theoretical classification problems. In particular they reformulated approximate transitivity in terms of matrix valued random walks and proved that a fixed point factor under an involutory automorphism acting on a bounded ITPFI factor yields an ITPFI factor. They also managed to provide an interesting example of an ergodic AT($2$) nonsingular transformation which is not AT.  

In this paper we study the flow of weights of factors arising as fixed point algebras under product type actions of finite groups on Araki-Woods factors. 
By imposing certain conditions, we obtain a concrete description of their associated flow.

In Section 2 we consider an action of $\mathbb{Z}_3$ on an ITPFI$_2$ factor, we compute the associated flow and prove that it is approximately transitive. The same methods can be applied for product type actions of other finite groups on bounded Araki-Woods factors. In Section 3, we construct a family of fixed point factors arising as fixed point algebras under product type actions of $\mathbb{Z}_n$ which are ITPFI factors and their flow of weights are built under constant functions and have aproduct type odometers as base transformations. 

The methods developed in \cite{GH} are quite efficient in the study of ergodic transformation. Applying these techniques to study the the associated flow of the factors considered in this paper, may lead to a better understanding of the way in which the properties of ergodic transformations are translated in terms of random walks. 

\section{Fixed point subfactors of bounded Araki-Woods factors}

We recall that an ITPFI factor is said to be of bounded type if it can written as an infinite tensor product of full
matrix algebras of uniformly bounded size. It is proved in \cite{GS1} that any factor of bounded type is an ITPFI$_2$ factor. 
In this section we analyze in details the flow of weights of factors obtained as fixed point algebras under product type action of $\mathbb{Z}_3$ on ITPFI$_2$ factors. The same methods can be used when we deal with product actions of other finite groups, but the computations may be more difficult.


Let $(l_{n})_{n\geq 1}$ be an increasing sequence of
positive integers, $(\lambda_{n})_{n\geq 1}$ a sequence of positive real numbers decreasing to zero  such that $l_1=1$, 
\begin{equation}\label{er}
\sum_{n=1}^\infty l_{n}\lambda_{n}=\infty
\end{equation}
and there exists $\epsilon_0>0$ such that 
\begin{equation}\label{lac}
-\log\lambda_{n}+2\sum_{i=1}^{n-1}l_{i}\log\lambda_{i}>\epsilon_0.
\end{equation}
Consider the ITPFI factor $M=\otimes_{n\geq 1}(M_2(\mathbb{C}), \varphi_n)^{\otimes 2l_n}$, where the states
$\varphi_n$ on $M_2(\mathbb{C})$ are given by the density matrix $\frac{1}{1+\lambda_n}\left(
\begin{array}{cc}
1 & 0 \\
0 & \lambda_i \\
\end{array}
\right)$. On $M$ we act by the product type action of $\mathbb{Z}_3$, corresponding to the automorphism
$\alpha= \otimes_{n\geq 1} Ad \left( \begin{array}{ccc}
1 & 0  \\
0 & \epsilon 
\end{array}
\right)$, where $\epsilon=e^{\frac{2\pi i}{3}}$. We denote by $N$ the fixed point algebra under this action that is
$$N=\{x\in M : \alpha(x)=x\}.$$

The first condition will ensure that $N$ is a factor of type III (see [12]). The second condition is a noninteractive one that will allow us to give an explicit description of the flow of weights of $N$. We will, prove the following result: 

\begin{theorem}
The flow of weights of $N$ is AT and $N$ is an ITPFI factor.
\end{theorem}
We proceed now to the computation of the flow of weights of $N$. Let us consider the product space $X=\prod_{n\geq 1}\{0,1\}$ endowed with the product measure $\nu=\otimes_{n\geq 1}\nu_{\lambda_n}^{\otimes 2l_n}$ where $\nu_{\lambda_n}(0)=\frac{1}{1+\lambda_n}$ and $\nu_{\lambda_n}(1)=\frac{\lambda_n}{1+\lambda_n}$. We recall that tail equivalence relation on $(X,\nu)$, denoted by $\mathcal{T}$, is defined for $x=(x_k)_{k \geq 1}$ and $y=(y_i)_{i\geq 1}$ by 
\[x\mathcal{T}y\text{ if and only if there exist }m\geq 1\text{ such that }x_k=y_k\text{ for }k>m.\]
It is well known that $M$ can be identified with  $M(\mathcal{T})$, the von Neumann algebra associated to the tail equivalence relation via the Feldmann-Moore construction.
Let $\mathcal{R}$ be the hyperfinite equivalence relation on $(X,\nu)$, given by
\[x\mathcal{R}y\text{ if and only if there exist } x\mathcal{T} y\text{ and }\sum_{i=1}^\infty(x_k-y_k)=0 \ (\text{mod }3).\]
Following [13] the fixed point algebra $N$ is isomorphic to  $M(\mathcal{R})$, the von Neumann algebra corresponding to $\mathcal{R}$. Since $M(\mathcal{T})$
is of type III$_0$, the assymptotic ratio set of $\mathcal{T}$ is $\{0,1\}$. Then, as $\mathcal{R}$ is a type III subequivalence relation of $\mathcal{T}$ (because $N$ is a factor of type III) it follows that $\mathcal{R}$ and the corresponding factor $N$ are of type III$_0$.

Thus, the flow of weights of $N$ is (up to conjugacy) the associated flow of the equivalence relation $\mathcal{R}$. This is given by the action of $\mathbb{R}$ by translation on $X\times\mathbb {R}/\mathcal{\widetilde{R}}$ (the quotient stands for ergodic decomposition), where $\mathcal{\widetilde{R}}$ is the equivalence relation on $(X\times\mathbb {R}, \nu\times m)$ defined by $(x,t) \mathcal{\widetilde{R}} (y,s)$ if a nd only if $x\mathcal{R}y$ and $s=t-\log\delta(x,y)$, where $\delta(x,y)$ is the Radon-Nicodym cocycle of $\nu$ with respect to $\mathcal{R}$ and $m$ is a finite measure on $\mathbb{R}$ equivalent to the Lebesgue measure.

Let $\widehat{X}=\prod\limits_{n\geq 1}\widehat{X}_{n}$, where $\widehat{X}_{1}=\{0,1,2\}$, and
$\widehat{X}_{n}=\{0,1,\ldots,2l_{n}\}$, $n\geq 2$ with the product
measure $\widehat{\nu}=\otimes\widehat{\nu}_{n}$ where,
for $n\geq 2$,
\[\widehat{\nu}_{n}(i)= \left(\begin{array}{c}
   2l_{n} \\
   i
 \end{array}
 \right) \frac{\lambda_{n}^{i}}{(1+\lambda_{n})^{2l_{n}}}, \ \ \ 0\leq i \leq
2l_{n}.\]

Let $(X_n,\nu_n)=\left(\{0,1\}, \nu_{\lambda_n}\right)^{\otimes 2l_n}$. Notice that $(X,\nu)$ can be naturally identified with the product space $\prod_{n\geq 1}(X_n, \nu_n)$. 
Let $\mathcal{S}_n$ be the equivalence relation on $(X_n,\nu_n)$ defined by $x\mathcal{S}_n y$ if and only if $\nu_n(x)=\nu_n(y)$. It is then easy to see that $(X_n,\nu_n)/\mathcal{S_n}=(\widehat{X}_{n},\widehat{\nu}_{n})$
where the projection map  
$\pi_n : X_n\rightarrow\widehat{X}_n$ is given by 
$\pi_n(x)=\sum\limits_{i=1}^{2l_n}x_i$ for $x=(x_1,x_2,\ldots, x_{2l_n})\in X_n.$
We further consider the
equivalence relation $\mathcal{S}$ on $(X,\nu)$ given by \[x\mathcal{S}y\text{ if and	only if } x\mathcal{R} y \text{ and }\log\delta(x,y)=0.\]
Clearly $\mathcal{S}$ is a subequivalence relation of $\mathcal{R}$. If  $x=(x_n)_{n\geq 1}, y=(y_n)_{n\geq 1}\in X=\prod_{n\geq 1}X_n$, it is easy to see from that
$$x\mathcal{S}y\text{ if and only if }x_n\mathcal{S}_ny_n \text{ for all }n\geq 1.$$
and therefore (see for example [14]) we have 
\begin{align*}\text{$
(X,\nu)/\mathcal{S}=\prod_{n=1}^\infty(X_n,\nu_n)/\mathcal{S}_n)=  \prod_{n= 1}^\infty(\widehat{X}_{n},\widehat{\nu}_{n})=(\widehat{X},\widehat{\nu})$}
\end{align*}
The quotient map $\pi: X\rightarrow \widehat{X}$ is given for $x=(x_n)_{n\geq 1}\in X=\prod_{n\geq 1} X_n$ by 
$$\pi(x_1,x_2,\ldots )=(\pi_1(x_1),\pi_2(x_2),\ldots )$$

Notice that by (2), the measure $\nu$ is lacunary. For $x\in X$ let $$\xi(x)=\min\{\log\delta(y,x);\ (y,x)\in\mathcal{R},\
\log\delta(y,x)>0\}$$ 
Since $\xi(x)=\xi(y)$ if $x\mathcal{S}y$ we can regard $\xi$ as a function on the quotient space $\widehat{X}$. We get an ergodic automorphism
$T$ defined on $\widehat{X}$ by setting $T(\pi(x))=\pi(x')$ where
$x\mathcal{R}x'$ and $\log \delta(x',x)=\xi(\pi(x))$.
Let $$\Omega=\{(z,t) :  z\in \widehat{X},  0\leq t< \xi(y)\}$$
and denote by $\mu$ the restriction of the product measure $\widehat{\nu}\otimes \lambda$, where $\lambda$ is the Lebesgue measure on $\mathbb{R}$.  
Then, the associated flow $\{F_{s}\}_{s\in\mathbb{R}}$ of
$\mathcal{R}$ is the flow built under the ceiling
function $\xi$ with base automorphism $T$ (see for example \cite{HO} or \cite{KW}), that is $\{F_s\}_{s\in\mathbb{R}}$ is defined on $(\Omega,\mu)$ as follows:
\begin{equation*}\label{flow}
F_{s}(z,t)=\left\{
\begin{array}{l}
(z,t+s) \text{ if} \ \ 0 \leq s +t< \xi(y)\\[0.2cm]
(T(z),t+s-\xi(z)) ,\text{ if}  \ \ \xi(y)\leq s +t< \xi(T(z))+\xi(z)\\[0.2cm]
\cdots
\end{array}
\right. \end{equation*}

\begin{lemma}\label{fl4}
	Let $x,x'\in X$, $x\mathcal{R} x'$,
	$s=\log\delta(x',x)$, $z=\pi(x)$, $z'=\pi(x')$. If $(z,t)$ and $(z',t)$ are in $\Omega$ then
	$F_{s}(z,t)=(z',t)$.
\end{lemma}
\begin{proof}
Lat us assume first that  $s=\log\delta(x',x)$ is positive. Since the measure $\nu$ is lacunary, there exist finitely many values, say $n$, of $\log\delta(z,x)$ between $0$ and $\log\delta(x',x)$. Corresponding to these distinct values we can find $x_{1}, x_{2},\ldots, x_{n}$ in the $\mathcal{R}$-orbit of $x$ such that
$0<\log\delta(x_{1},x)<\cdots<\log\delta(x_{n},x)<\log\delta(x',x)$.
Then $$\log\delta(x',x)=\log\delta(x_{1},x)+\log\delta(x_{2},x_{1})+\cdots+\log\delta(x_{n},x_{n-1})+\log\delta(x',x_{n}).$$
If $z_{k}=\pi(x_{k})$, we have $z_{k}=T^{k}(z)$ and $F_{\xi(T^{k-1}z)}(T^{k-1}(z),t)=(T^{k}(z),t)$, for $1\leqslant k\leqslant n$. As $\log\delta(x',x)=\xi(z)+\xi(T(z))+\cdots+\xi(T^{n}(z))$. we have\begin{align*}
F_{s}(z,t)&=F_{\xi(z)+\xi(T(z))+\cdots+\xi(T^{n}(z))}(z,t)\\
&=F_{\xi(T(z))+\cdots+\xi(T^{n}(z))}(T(z),t)=\cdots=(z',t).
\end{align*}	
If $s<0$, then $-s>0$. As above, $F_{-s}(z',t)=(z,t)$ and
therefore $F_{s}(z,t)=F_{s}(F_{-s}(z,t))=(z',t)$.
\end{proof}
Notice that if
$z=(z_{1},z_{2},\ldots z_{m},z_{m+1}\ldots)$ and for some $n\in
\mathbb{Z}$, $T^{n}(z)=z'$ with $z=(z'_{1},z'_{2},\ldots,
z'_{m},z_{m+1},\ldots)$ then \[\frac{d\widehat{\nu}\circ T^{n}}{d\widehat{\nu}}(z)=\prod_{i=1}^{m}\frac{\widehat{\nu}_{i}(z'_{i})}{\widehat{\nu_{i}}(z_{i})}.\]

\noindent The following proposition] is immediate consequence of the above remark and Lemma \ref{fl4}.
\begin{proposition}\label{fl6}
	Let $z,z'\in \widehat{X}$ and let
	$x,x'\in X$, $x\mathcal{R}x'$ such that $z=\pi(x)$, $z'=\pi(x')$ and $x_{i}=x'_{i}$ if $i>2\sum\limits_{i=1}^m l_{i}$. If $(z,t), (z',t)\in \Omega$ and  $s=\log\delta(x',x)=\sum\limits_{i=1}^{m}(z'_i-z_i)\log \lambda_i$ then
	\[F_{s}(z,t)=(z',t)\text{ and }\frac{d\mu\circ
		F_{s}}{d\mu}(z,t)=\prod_{i=1}^{m}\frac{\widehat{\nu}_{i}(z'_{i})}{\widehat{\nu_{i}}(z_{i})}.\]
\end{proposition}
Let $n>1$. Let us introduce some notations. If $a_i\in \widehat{X}_i$ for $i=1,2,\ldots,n$ let
$$C(a_1,a_2,\ldots, a_n)=\{y\in \widehat{X} :
y_{i}=a_i, i=1,2,\ldots,n\}.$$
In praticular we define
\begin{align*}
C(0_n)&=\{y\in \widehat{X} : y_i=0, i=1,2,\ldots,n \}
\end{align*}
and for $a\in\{1,2\}$ 
$$C(a,0_{n-1})=\{y\in\widehat{X} : y_1=a, y_i=0, i=2,3,\ldots,n \}.$$
Let $\mathbb{R}[X_{n+1},X_{n+2},\ldots, X_{m}]$ be the ring
of polynomials in $X_{n+1}, X_{n+2},\ldots, X_{m}$ with real
coefficients. For a polynomial $P=\sum a_{i_{n+1},i_{n+2},\ldots,
	i_{m}}X_{n+1}^{k_{n+1}}X_{n+2}^{k_{n+2}}\cdots X_{m}^{k_{m}}$ in $\mathbb{R}[X_{n+1},X_{n+2},\ldots,X_{m}]$
we set
\[\|P\|=\sum |a_{i_{n+1},i_{n+2},\ldots, i_{m}}|.\]
It is easy to see that $P\mapsto\|P\|$ defines a norm and if $P$ and $Q$ are
two polynomials from $\mathbb{R}[X_{n+1},X_{n+2},\ldots X_{m}]$ we have
\[\|P\cdot Q\|\leq \|P\|\cdot\|Q\|\]
For $j\in\{0,1,2\}$ let 
\begin{align*}
I_{j}^{m}&=\{(i_{n+1},i_{n+2},\ldots, i_{m}): 0\leq i_{k}\leq
l_{k},\  n< k\leq m, \ i_{n+1}+ i_{n+2}+\cdots+
i_{m}=j\text{ (mod 3)}\}
\end{align*}
and
\begin{align*}
P_{j}^{m}&=\{(p_{n+1},p_{n+2},\ldots, p_{m}): 0\leq p_{k}\leq
2l_{k},\  n< k\leq m, \ p_{n+1}+ p_{n+2}+\cdots+
p_{m}=j\text{ (mod 3)}\}
\end{align*}
Let $I^m=I_{0}^{m}\cup
I_{1}^{m}\cup
I_{2}^{m}$ and $P^m=P_{0}^{m}\cup
P_{1}^{m}\cup
P_{2}^{m}$. We denote by $\overline{i}$ an element $(i_{n+1},i_{n+2},\ldots, i_{m})\in I^{m}$ and by $\overline{p}$ an element $(p_{n+1},p_{n+1},\ldots, p_{m})\in P^m$. Let $\epsilon$ be the primitive $3$rd root of unity. With this notation, we have
\begin{align}
\prod_{i=n+1}^{m}&\left(\frac{1}{1+\lambda_{i}}+\epsilon \frac{\lambda_{i}}{1+\lambda_{i}}X_{i}\right)^{l_{i}}
\left(\frac{1}{1+\lambda_{i}}+\overline{\epsilon}\frac{\lambda_{i}}{1+\lambda_{i}}X_{i}\right)^{l_{i}}=\nonumber\\
&=\sum_{\overline{j}\in I_{0}^{m},\overline{k}\in
	I_{0}^{m}}\prod_{i=n+1}^{m}\binom{l_i}{j_i}\binom{l_i}{k_i}\frac{\lambda_{i}^{j_{i}+k_{i}}}{(1+\lambda_{i})^{2
		l_{i}}}X_{i}^{k_{i}+j_{i}}+\sum_{\overline{j}\in
	I_{1}^{m},\overline{k}\in
	I_{1}^{m}}\prod_{i=n+1}^{m}\binom{l_i}{j_i}\binom{l_i}{k_i}\frac{\lambda_{i}^{j_{i}+k_{i}}}{(1+\lambda_{i})^{2
		l_{i}}}X_{i}^{k_{i}+j_{i}}\nonumber\\
&+\sum_{\overline{j}\in
	I_{2}^{m},\overline{k}\in
	I_{2}^{m}}\prod_{i=n+1}^{m}\binom{l_i}{j_i}\binom{l_i}{k_i}\frac{\lambda_{i}^{j_{i}+k_{i}}}{(1+\lambda_{i})^{2
		l_{i}}}X_{i}^{k_{i}+j_{i}}-\\ \label{r}
&-\sum_{\overline{j}\in
	I_{0}^{m},\overline{k}\in
	I_{1}^{m}}\prod_{i=n+1}^{m}\binom{l_i}{j_i}\binom{l_i}{k_i}\frac{\lambda_{i}^{j_{i}+k_{i}}}{(1+\lambda_{i})^{2
		l_{i}}}X_{i}^{k_{i}+j_{i}}-\sum_{\overline{j}\in
	I_{0}^{m},\overline{k}\in
	I_{2}^{m}}\prod_{i=n+1}^{m}\binom{l_i}{j_i}\binom{l_i}{k_i}\frac{\lambda_{i}^{j_{i}+k_{i}}}{(1+\lambda_{i})^{2
		l_{i}}}X_{i}^{k_{i}+j_{i}}\nonumber\\
&-\sum_{\overline{j}\in
	I_{1}^{m},\overline{k}\in
	I_{2}^{m}}\prod_{i=n+1}^{m}\binom{l_i}{j_i}\binom{l_i}{k_i}\frac{\lambda_{i}^{j_{i}+k_{i}}}{(1+\lambda_{i})^{2
		l_{i}}}X_{i}^{k_{i}+j_{i}} \nonumber
\end{align}
Notice that
\begin{align}
\lim_{m\rightarrow\infty}&\left\|\prod_{i=n+1}^{m}\left(\frac{1}{1+\lambda_{i}}+\epsilon\frac{\lambda_{i}}{1+\lambda_{i}}X_{i}\right)^{l_{i}}
\left(\frac{1}{1+\lambda_{i}}+\overline{\epsilon}\frac{\lambda_{i}}{1+\lambda_{i}}X_{i}\right)^{l_{i}}\right\|\nonumber\\
& \ \ \leq\lim_{m\rightarrow\infty}\prod_{i=n+1}^{m}\left\|\frac{1}{(1+\lambda_i)^2}+\frac{\lambda_{i}^{2}}{(1+\lambda_{i})^2}X_{i}^{2}-\frac{\lambda_{i}}{(1+\lambda_{i})^2}X_{i}\right\|^{l_{i}}\\ \label{rr} 
& \ \ =\lim_{m\rightarrow\infty}\prod_{i=n+1}^{m}\left|1-\frac{\lambda_{i}}{(\lambda_{i}+1)^{2}}\right|^{l_{i}}=0 \nonumber
\end{align}
since
\[\lim_{m\rightarrow\infty}\sum_{i=n+1}^{m}l_{i}\lambda_{i}=\infty.\]
Note that if $p_{n+1}+p_{n+2}+\cdots+p_m=0$ (mod $3$) then   
\begin{align}
\sum_{\overline{j}\in I_{0}^{m},\overline{k}\in
	I_{0}^{m}, j_i+k_i=p_i}\prod_{i=n+1}^{m}\binom{l_i}{j_i}\binom{l_i}{k_i}\frac{\lambda_{i}^{j_{i}+k_{i}}}{(1+\lambda_{i})^{2
		l_{i}}}+ 2\sum_{\overline{j}\in
	I_{1}^{m},\overline{k}\in
	I_{2}^{m},j_i+k_i=p_i}\prod_{i=n+1}^{m}\binom{l_i}{j_i}\binom{l_i}{k_i}\frac{\lambda_{i}^{j_{i}+k_{i}}}{(1+\lambda_{i})^{2
		l_{i}}}\nonumber\\
=\prod_{i=n+1}^{m}\binom{2l_i}{p_i}\frac{\lambda_{i}^{2p_{i}}}{(1+\lambda_{i})^{2
		l_{i}}},
\end{align}
if $p_{n+1}+p_{n+1}+\cdots+p_m=1$ (mod $3$) we have   
\begin{align}
\sum_{\overline{j}\in I_{1}^{m},\overline{k}\in
	I_{1}^{m}, j_i+k_i=p_i}\prod_{i=n+1}^{m}\binom{l_i}{j_i}\binom{l_i}{k_i}\frac{\lambda_{i}^{j_{i}+k_{i}}}{(1+\lambda_{i})^{2
		l_{i}}}+ 2\sum_{\overline{j}\in
	I_{0}^{m},\overline{k}\in
	I_{2}^{m},j_i+k_i=p_i}\prod_{i=n+1}^{m}\binom{l_i}{j_i}\binom{l_i}{k_i}\frac{\lambda_{i}^{j_{i}+k_{i}}}{(1+\lambda_{i})^{2
		l_{i}}}\nonumber\\
=\prod_{i=n+1}^{m}\binom{2l_i}{p_i}\frac{\lambda_{i}^{2p_{i}}}{(1+\lambda_{i})^{2
		l_{i}}},
\end{align}
and if 
if $p_{n+1}+p_{n+2}+\cdots+p_m=2$ (mod $3$) we have   
\begin{align}
\sum_{\overline{j}\in I_{2}^{m},\overline{k}\in
	I_{2}^{m}, j_i+k_i=p_i}\prod_{i=n+1}^{m}\binom{l_i}{j_i}\binom{l_i}{k_i}\frac{\lambda_{i}^{j_{i}+k_{i}}}{(1+\lambda_{i})^{2
		l_{i}}}+ 2\sum_{\overline{j}\in
	I_{0}^{m},\overline{k}\in
	I_{1}^{m},j_i+k_i=p_i}\prod_{i=n+1}^{m}\binom{l_i}{j_i}\binom{l_i}{k_i}\frac{\lambda_{i}^{j_{i}+k_{i}}}{(1+\lambda_{i})^{2
		l_{i}}}\nonumber\\
=\prod_{i=n+1}^{m}\binom{2l_i}{p_i}\frac{\lambda_{i}^{2p_{i}}}{(1+\lambda_{i})^{2
		l_{i}}}, \label{rrr}
\end{align}
Then, by (\ref{r}) - (\ref{rrr}) it foloows that for $m\longrightarrow\infty$ we have
\begin{eqnarray}\label{uno}
\sum_{\overline{p}\in P_0^m}
\left|\sum_{\overline{j}\in I_{0}^{m}, \overline{k}\in I_{0}^{m},
	k_{i}+j_{i}=p_{i}}3\prod_{i=n+1}^{m}\binom{l_i}{j_i}\binom{l_i}{k_i}-\prod_{i=n+1}^{m}
\binom{2l_i}{p_i}\right|\prod_{i=n+1}^{m}\frac{\lambda_{i}^{p_{i}}}{(1+\lambda_{i})^{2 l_{i}}}
\rightarrow 0,
\end{eqnarray}
\begin{eqnarray}\label{due}
\sum_{\overline{p}\in P_2^m}
\left|\sum_{\overline{j}\in I_{1}^{m}, \overline{k}\in I_{1}^{m},
	k_{i}+j_{i}=p_{i}}3\prod_{i=n+1}^{m}\binom{l_i}{j_i}\binom{l_i}{k_i}-\prod_{i=n+1}^{m}
\binom{2l_i}{p_i}\right|\prod_{i=n+1}^{m}\frac{\lambda_{i}^{p_{i}}}{(1+\lambda_{i})^{2 l_{i}}}
\rightarrow 0,
\end{eqnarray}
\begin{eqnarray}\label{tres}
\sum_{\overline{p}\in P_1^m}
\left|\sum_{\overline{j}\in I_{2}^{m}, \overline{k}\in I_{2}^{m},
	k_{i}+j_{i}=p_{i}}3\prod_{i=n+1}^{m}\binom{l_i}{j_i}\binom{l_i}{k_i}-\prod_{i=n+1}^{m}
\binom{2l_i}{p_i}\right|\prod_{i=n+1}^{m}\frac{\lambda_{i}^{p_{i}}}{(1+\lambda_{i})^{2 l_{i}}}
\rightarrow 0,
\end{eqnarray}
\begin{eqnarray}\label{quatro}
\sum_{\overline{p}\in P_0^m}
\left|\sum_{\overline{j}\in I_{1}^{m},\overline{k}\in I_{2}^{m},
	k_{i}+j_{i}=p_{i}} 3\prod_{i=n+1}^{m}
\binom{l_i}{j_i}\binom{l_i}{k_i}- \prod_{i=n+1}^{m}C_{2
	l_{i}}^{p_{i}}\right|\prod_{i=n+1}^{m}\frac{\lambda_{i}^{p_{i}}}{(1+\lambda_{i})^{2
		l_{i}}}\rightarrow 0,
\end{eqnarray}
\begin{eqnarray}\label{cinque}
\sum_{\overline{p}\in P_2^m}
\left|\sum_{\overline{j}\in I_{0}^{m},\overline{k}\in I_{2}^{m},
	k_{i}+j_{i}=p_{i}} 3\prod_{i=n+1}^{m}
\binom{l_i}{j_i}\binom{l_i}{k_i}- \prod_{i=n+1}^{m}C_{2
	l_{i}}^{p_{i}}\right|\prod_{i=n+1}^{m}\frac{\lambda_{i}^{p_{i}}}{(1+\lambda_{i})^{2
		l_{i}}}\rightarrow 0,
\end{eqnarray}
\begin{eqnarray}\label{six}
\sum_{\overline{p}\in P_1^m}
\left|\sum_{\overline{j}\in I_{0}^{m},\overline{k}\in I_{1}^{m},
	k_{i}+j_{i}=p_{i}} 3\prod_{i=n+1}^{m}
\binom{l_i}{j_i}\binom{l_i}{k_i}- \prod_{i=n+1}^{m}C_{2
	l_{i}}^{p_{i}}\right|\prod_{i=n+1}^{m}\frac{\lambda_{i}^{p_{i}}}{(1+\lambda_{i})^{2
		l_{i}}}\rightarrow 0,
\end{eqnarray}
For $\overline{k}=(k_{n+1},\ldots, k_{m})\in I^m$, we set
$$s(\overline{k})=\sum\limits_{i=n+1}^{m}k_{i}\log\lambda_{i}.$$ 
and if $a=(a_1,a_2,\ldots,a_n)$ define $r(a)\in\{0,1,2\}$ by 
$$r(a)=\sum_{i=1}^n a_{i} \ (\text{mod }3)$$
and
$$s(a)=\sum_{i=1}^n a_i\log\lambda_i.$$
For $a=(a_1,a_2,\ldots,a_n)$ and $\overline{k}=(k_{n+1},\ldots, k_{m})\in P^m$
let
$$C(a,\overline{k})=C(a)\cap \{y\in \widehat{X} : y_i=k_i,  \text{for }i=n+1,n+2,\ldots, m\}.$$
Let $a=(a_1,a_2,\ldots, a_n)$ with $r(a)=2$, $\overline{j}\in I^m$, $\overline{k}=I^m_1$ and $z\in C(0_n,\overline{j})$. Let $z'\in C(a,\overline{j}+\overline{k})$ and  such that $z'_{i}=z_{i}$ for $i>m$. We can find $x$
and $x'$ in $X$ such that $x\mathcal{R} x'$, $\pi(x)=z$, $\pi(x')=z'$ and
$\log\delta(x',x)=s(\overline{k})=\sum\limits_{i=n+1}^{m}k_{i}\log \lambda_{i}$. Then, by
Lemma \ref{fl4},
\begin{align*}
F_{s(\overline{k})}(z,t)=(z',t) \text{ and }F_{-s(\overline{k})}(z',t)=(z,t).
\end{align*}

Thus for $\overline{k}\in P_1^m$ we have
\begin{equation*}
F_{s(\overline{k})}\left(C(0_n, \overline{j})\right)\\
= C(a,\overline{j}+\overline{k})
\end{equation*}
and by Proposition \ref{fl6}
\begin{equation*}
\frac{d\mu
	F_{-s(\overline{k})}}{d\nu}(z',t)=\prod_{i=n+1}^{m}\frac{\nu_{i}(j_{i})}{\nu_{i}(j_{i}+k_{i})}\cdot \prod_{i=1}^{n}\frac{1}{\nu_{i}(a_{i})}=\prod_{i=n+1}^{m}\frac{\binom{2l_i}{j_i}}{\binom{2l_i}{j_i+k_i}}
\cdot\frac{1}{\lambda_{i}^{k_{i}}}\cdot\prod_{i=1}^n\frac{1}{\binom{2l_i}{a_i}}\cdot\frac{1}{\lambda_{i}^{a_{i}}},
\end{equation*}
for $(z',t)\in C(a,\overline{j}+\overline{k})$. Therefore
\begin{align*}
&&\chi_{C(0_{n},\overline{j})\times[0,\delta]}(F_{-s(\overline{k})})\frac{d\mu F_{-s((\overline{k}))}}{d\mu}=\prod_{i=n+1}^{m}\frac{\binom{2l_i}{j_i}}{\binom{2l_i}{j_i+k_i}}\cdot\frac{1}{\lambda_{i}^{k_{i}}}\cdot \prod_{i=1}^n\frac{1}{\binom{2l_i}{a_i}}\cdot\frac{1}{\lambda_{i}^{a_{i}}}\cdot\chi_{C(a,\overline{j}+\overline{k})\times[0,\delta]} 
\end{align*}
Let $n>2$ and $0<\delta<\epsilon_0$ (such that $\overline{X}\times
[0,\delta]\subseteq \Omega$). We define
\begin{align*}
f=\sum_{\overline{j}\in I^m}
\prod_{i=n+1}^{m}\frac{\binom{l_i}{j_i}}{\binom{2l_i}{j_i}}\chi_{C(0_{n},\overline{j})\times[0,\delta]}
\end{align*}
We have
\begin{align*}
\prod_{i=1}^n\binom{2l_i}{a_i}&\lambda_{i}^{a_{i}}\sum_{\overline{k}\in
	I_{1}^{m}}\prod_{i=n+1}^{m}  \lambda_{i}^{k_{i}}\binom{l_i}{k_i}\cdot f(F_{-s(\overline{k})})\frac{\mu F_{-s(\overline{k})}}{d\mu}   =\sum_{\overline{k}\in I_{1}^{m}}\sum_{\overline{j}\in I^m_0\cup I^m_1\cup I^m_2} \prod_{i=n}^{m}\frac{\binom{l_i}{k_i} \binom{l_i}{j_i}}{\binom{2l_i}{j_i+k_i}}\chi_{C(a,  \overline{j}+\overline{k})\times[0,\delta]}\\
=&\sum_{\overline{k}\in I_{1}^{m},\overline{j}\in
	I_{0}^{m}}\prod_{i=n+1}^{m}\frac{\binom{l_i}{k_i} \binom{l_i}{j_i}}{\binom{2l_i}{j_i+k_i}}
\chi_{C(a, \overline{j}+\overline{k})\times[0,\delta]}
+\sum_{\overline{k}\in I_{1}^{m}, \ \overline{j}\in
	I_{1}^{m}}\prod_{i=n+1}^{m}\frac{\binom{l_i}{k_i} \binom{l_i}{j_i}}{\binom{2l_i}{j_i+k_i}}
\chi_{a, \overline{j}+\overline{k})\times[0,\delta]}+\\
&\sum_{\overline{k}\in I_{1}^{m},\overline{j}\in
	I_{2}^{m}}\prod_{i=n+1}^{m}\frac{\binom{l_i}{k_i}\binom{l_i}{j_i}}{C_{2l_{i}}^{j_{i}+k_{i}}}\chi_{C(a, \overline{j}+\overline{k})\times[0,\delta]}\\
=&\sum_{\overline{p}\in P_1^m} \ \sum_{\substack{\overline{k}\in
		I_{1}^{m},\overline{j}\in I_{0}^{m}\\
		j_{i}+k_{i}=p_{i}}}\ \prod_{i=n+1}^{m}\frac{\binom{l_i}{k_i}\binom{l_i}{j_i}}{\binom{2l_i}{p_i}}\chi_{C(a,\overline{p} )\times[0,\delta]}
+\sum_{\overline{p}\in P_2^m} \ \sum_{\substack{\overline{k}\in
		I_{1}^{m},\overline{j}\in I_{0}^{m}\\
		j_{i}+k_{i}=p_{i}}}\ \prod_{i=n+1}^{m}\frac{\binom{l_i}{k_i}\binom{l_i}{j_i}}{\binom{2l_i}{p_i}}\chi_{C(a,\overline{p})\times[0,\delta]}+\\
&\sum_{\overline{p}\in P_0^m} \ \sum_{\substack{\overline{k}\in
		I_{1}^{m},\overline{j}\in I_{2}^{m}\\
		j_{i}+k_{i}=p_{i}}}\ \prod_{i=n+1}^{m}\frac{\binom{l_i}{k_i}\binom{l_i}{j_i}}{\binom{2l_i}{p_i}}\chi_{C(a,\overline{p})\times[0,\delta]}
\end{align*}
Therefore,
\begin{align*}
\left\|3\prod_{i=1}^n\right.& \left.\binom{2l_i}{a_i}\lambda_{i}^{a_{i}}\cdot\sum_{\overline{k}\in
	I_{1}^{m}}    \prod_{i=n+1}^{m}\binom{l_i}{k_i}\lambda_{i}^{k_{i}}\cdot  f(F_{-s(\overline{k})})\frac{\mu F_{-s(\overline{k})}}{d\mu}-\chi_{C(a)\times[0,\delta]}\right\|\\
&\leq \sum_{\overline{p}\in P_{1}^m}\left\|\left(\sum_{\overline{k}\in I_{1}^{m}, \overline{j}\in
	I_{0}^{m}, 
	j_{i}+k_{i}=p_{i}}3\prod_{i=n+1}^{m}\frac{\binom{l_i}{k_i}\binom{l_i}{j_i}}{\binom{2l_i}{p_i}}-1\right)\chi_{C(a,\overline{p})\times[0,\delta]}\right\|\\
&+\sum_{\overline{p}\in P^m_2}\left\|\left(\sum_{\overline{k}\in I_{1}^{m}, \overline{j}\in
	I_{1}^{m}, 
	j_{i}+k_{i}=p_{i}}3\prod_{i=n+1}^{m}\frac{\binom{l_i}{k_i}\binom{l_i}{j_i}}{\binom{2l_i}{p_i}}-1\right)\chi_{C(a, \overline{p})\times[0,\delta]}\right\|\\
&+\sum_{\overline{p}\in P^m_0}\left\|\left(\sum_{\overline{k}\in I_{1}^{m}, \overline{j}\in
	I_{2}^{m}, 
	j_{i}+k_{i}=p_{i}}3\prod_{i=n+1}^{m}\frac{\binom{l_i}{k_i}\binom{l_i}{j_i}}{\binom{2l_i}{p_i}}-1\right)\chi_{C(a, \overline{p})\times[0,\delta]}\right\|\\
&=\sum_{\overline{p}\in P^m_1}\left|\sum_{\overline{k}\in I_{1}^{m}, \overline{j}\in
	I_{0}^{m},
	j_{i}+k_{i}=p_{i}}3\prod_{i=n+1}^{m}\binom{l_i}{k_i}\binom{l_i}{j_i}-\prod_{i=n+1}^{m}C_{2
	l_{i}}^{p_{i}}\right|\prod_{i=n+1}^{m}\frac{\lambda_{i}^{p_{i}}}{(1+\lambda_{i})^{2 l_{i}}}\|\chi_{C(a)\times[0,\delta]}\|\\
&+\sum_{\overline{p}\in P^m_2}\left|\sum_{\overline{k}\in I_{1}^{m},\overline{j}\in I_{1}^{m},
	j_{i}+k_{i}=p_{i}}3\prod_{i=n+1}^{m}\binom{l_i}{k_i}\binom{l_i}{j_i}-\prod_{i=n+1}^{m}C_{2
	l_{i}}^{p_{i}}\right|\prod_{i=n+1}^{m}\frac{\lambda_{i}^{p_{i}}}{(1+\lambda_{i})^{2
		l_{i}}}\|\chi_{C(a)\times[0,\delta]}\|\\
&+\sum_{\overline{p}\in P^m_0}\left|\sum_{\overline{k}\in I_{1}^{m},\overline{j}\in I_{2}^{m},
	j_{i}+k_{i}=p_{i}}3\prod_{i=n+1}^{m}\binom{l_i}{k_i}\binom{l_i}{j_i}-\prod_{i=n+1}^{m}C_{2
	l_{i}}^{p_{i}}\right|\prod_{i=n+1}^{m}\frac{\lambda_{i}^{p_{i}}}{(1+\lambda_{i})^{2
		l_{i}}}\|\chi_{C(a)\times[0,\delta]}\|.
\end{align*}
Similar inequalities can be obtained if $r(a)\in\{0,1\}$. 
Let $\varepsilon>0$. Then, by (8)-(13), we can choose $m$ large enough such that for every $a=(a_1,a_2,\ldots,a_n)$ we have
$$\left\|3\prod_{i=1}^n\binom{2l_i}{a_i}\lambda_{i}^{a_{i}}\cdot\sum_{\overline{k}\in
	I_{s}^{m}}    \prod_{i=n+1}^{m}\binom{l_i}{k_i}\lambda_{i}^{k_{i}}\cdot  f(F_{-s(\overline{k})})\frac{\mu F_{-s(\overline{k})}}{d\mu}-\chi_{C(a)\times[0,\delta]}\right\|\leq 
\varepsilon\|\chi_{C(a)}\|$$
where $s\in\{0,1,2\}$ satisfies $r(a)+s=0$ (mod $3$). 
Since any function from $L^1_+(\Omega, \mu)$ function can be approximated arbitrary closed by linear combinations of funtions of the form $\chi_{C(a)\times[t,t+\delta}$ it follows that the associated flow is AT.  
	
\begin{remark}\rm{
The result remains true for other abelian groups acting by diagonal xerox type actions on bounded Araki-Woods factors, but the computations may become more complicated. For fixed point factors under actions of infinite groups the problem is more difficult. If \ $\lim_{n\rightarrow}l_n\lambda_n=\infty$ it was proved in \cite{BGM} that the natural action of $S_\infty$ (the group of finite permutations of $\mathbb{N}$) on $(X,\nu)$ and the associated flow of $S_\infty$ are AT. This result implies that the fixed point algebra under the standard action of the torus is an ITPFI factor. In the absence of the assumption that $\lim_{n\rightarrow}l_n\lambda_n=\infty$ we don't know the answer. }
\end{remark}

\section{A fanily of factors arising as fixed point factors under $\mathbb{Z}_n$ actions}	
	
In this section we construct a family of fixed point factors under actions of $\mathbb{Z}_{n}$ which are ITPFI factors and their associated flows
are (up to conjugacy) built under constant functions and have product type odometers as base transformations. We recall that, in  general, a factor arising as fixed point algebra under a product type action may not be ITPFI factor. In \cite{GH} it is given an example of a fixed point factor under an involutory automorphism on an unbounded ITPFI which is not an ITPFI factor.  

Let us consider two positive integers $\lambda, k\geq 2$. For $n\geq 1$, let $r_{n}=\lambda^{(k+1)^{n}}$. Define $k_{0}=k$ and for $n\geq 1$ let $k_{n}=r_{n}+r_{n}^{2}+\cdots r_{n}^{k}$. On $M_{k_{n}}(\mathbb{C})$ consider the diagonal states
$\varphi_{n}(\cdot)=$tr$(a^{n}\cdot)$, where $a^{0}=\frac{1}{k}I_{k}$ and
$$a^{n}=\frac{1}{k+1}\cdot\text{diag}(\underbrace{ \frac{1}{r_{n}^{k}}\cdots,\frac{1}{r_{n}^{k}}}_{r_{n}^{k}\text{ times }},\underbrace{\frac{1}{r_{n}^{k-1}},\cdots,\frac{1}{r_{n}^{k-1}}}_{r_{n}^{k-1}\text{ times }},\cdots \underbrace{\frac{1}{r_{n}},\cdots,\frac{1}{r_{n}}}_{r_{n}\text{ times}}, 1)$$.

Consider the ITPFI factor $M=\otimes_{n\geq 0}(M_{k_{n}}(\mathbb{C}),\varphi_{n})$
and the primitive root of unity $\epsilon=e^{2\pi i/n}$. On $M$  we act by the product action of $\mathbb{Z}_{k}$ corresponding to the automorphism
$\alpha=\underset{n\geq 0}\otimes\text{Ad }u_n$  of order $k$, where
\[u_0=\text{diag}\left(1,\epsilon,\epsilon^{2},\ldots, \epsilon^{k-1}\right), \]and, for $n\geq 1$,
\[u_n=\text{diag}\left(I_{r_{n}^{k}},\epsilon I_{r_{n}^{k-1}},\ldots,
\epsilon^{k-1}I_{r_{n}}, 1\right) .\] We denote by $N$ the fixed point von Neumann algebra under this action. On the product space $X=\prod_{n\geq
	0}\{0,1,,\ldots, k_{n}\}$ we consider the product measure $\nu=\otimes_{n\geq 0}\nu_{n}$ defined by $\nu_{n}(i)=a^{n}_{i+1,i+1}$ for $0\leq i\leq k_{n}$ and $n\geq 0$.
For $n\geq 1$, divide $X_{n}=\{0,1,\ldots,k_{n}\}$ in subsets $I_{n}^{j}$, where
$$I_{n}^{0}=\{i\in \mathbb{Z}; 0\leq i< r_{n}^{k}\}$$ and, for $1\leq j\leq k$
$$I_{n}^{j}=\{i\in \mathbb{Z}: \  r_{n}^{k-j+1}+r_{n}^{k-j+2}+\cdots+r_{n}^{k}\leq i< r_{n}^{k-j}+r_{n}^{k-j+1}+\cdots+r_{n}^{k}\}.$$
Let now $Z=\prod_{n\geq 0} Z_n$, where $Z_0=\{0,1,\ldots,k-1\}$ and $Z_n=\{0,1,\ldots,k\}$ for $n\geq 1$ endowed with the product measure $\mu=\otimes_{n\geq 1}\mu_{n}$, where
$\nu_{0}(i)=\frac{1}{k}$, for $0\leq i\leq k-1$ and $\mu_{n}(i)=\frac{1}{k+1}$ for $0\leq i \leq k$, $n\geq 1$. Define the projection $\pi:X\rightarrow Z$, by 
\[\pi(x_{1},x_{2},\ldots)=(z_{1},z_{2},\ldots),\] where $z_{1}=x_{1}$ and for $n\geq2$, $z_{n}=j$ if $x\in I_{n}^{j}$.

Let $\mathcal{R}$ be the equivalence relation on $X$ given by
\[x\mathcal{R} y\text{ if and only if }x\mathcal{T} y\text{ and }\sum\pi(x)_{i}-\pi(y)_{i}=0 \  (\text{mod }k)\]
where $\mathcal{T}$ is the tail equivalence on $X$. Then, proceeding as in \cite{GM},  we can show that $N$ is isomorphic to $M(\mathcal{R})$.

Let $A_{n}=\{(z,w)\in Z\times Z: z_{kn+i}=1, w_{kn+i}=0, \  i=1,2,\ldots,k \}$ for  $n\geqslant 1$. Then
$\sum\limits_{n=1}^{\infty}\mu\times\mu(A_{n})=\infty$. By Borel-Cantelli, for $\mu\times\mu$ - almost all $(z,w)\in Z\times Z$ we have $(z,w)\in
A_{n}$ for infinitely many $n$. But if $(z,w)\in A_{n}$ there exists $j\in\{1,2,\ldots,k\}$
such that $\sum\limits_{j=1}^{kn+i}(z_{j}-w_{j})=0 \ \ (\text{mod } 3)$. Therefore
for $\mu\times\mu$ - almost all $(z,w)$, there are infinitely many
$n\in\mathbb{N}$ such that $\sum\limits_{i=1}^{n}(z_{i}-w_{i})=0 \ \
(\text{mod }3)$.
Now as $\nu=\mu\circ\pi^{-1}$, it follows that $\nu\times\nu$ - a.e. $(x,y)\in X\times X$, there are infinitely many $n$ such that 
$$\sum_{i=1}^n \pi(x)_i-\pi(y)_i=0 \ (\text{mod }k).$$
By using the same type arguments as in \cite{M} it follows that $\mathcal{R}$ is ergodic and of type III and therefore $N$ is factor of type III.

We compute the flow of weights of $N$ and show that it is approximately transitive. In fact we prove more then this:

\begin{theorem}
The flow of weights of $N$ is, up to conjugacy, the flow built under the constant function $\log\lambda$ and has as base transformation the measure preserving product odometer defined on $Z=\prod_{n\geq 0} Z_n$, where $Z_0=\{0,1,\ldots,k-1\}$ and $Z_n=\{0,1,\ldots,k\}$ for $n\geq 1$.  
\end{theorem}
\begin{proof}
Let $S$ be the measure preserving automorphism defined on 
$(Z,\mu)$ by
$S(a,k,k,k,\ldots)=(b,0,0,0,\ldots)$ with $b=a-1$ (mod $k$) and, if $z\neq (a,k,k,k,\ldots)$
\begin{equation*}S(z)_{n}=\left\{\begin{array}{lcl}
a_{z} & \text{ if }& n=0,  \\
0 & \text { if } & n\geq 1, n<N(z), \\
z_{n}+1 & \text  { if }& n=N(z), \\
z_{n}  & \text{ if }& n> N(z),
\end{array}
\right.
\end{equation*}where $N(z)=\min\{n\geq 1; z_{n}<k\}$ and
$a_{z}\in\{0,1,\ldots,k-1\}$ satisfies
$a_{z}=z_{0}+z_{1}+z_{2}+\cdots+z_{N(z)-1}-1\text{
	(mod k)}$.
Then, proceeding as in Section 2, it follows that, up to conjugacy, the associated flow of $\mathcal{R}$ is the flow built under
the constant function $\log\lambda$ with base transformation $S$.
Notice that for $n\geq 0$ we have $$S(C(a,\underbrace{k,k,k,\ldots,k}_{n \text{ times}}))=C(b,\underbrace{0,0,0,\ldots,0}_{n \text { times}}),$$ where $b=a-1$ (mod $k$). Let $\Delta_{n}$ be the partition of $Z$ into cylinders of the form $C(a_{0}, a_{1},\ldots, a_{n})$. For $n\geq0$, let $A_{n}^{0}=\{z\in Z: z_{i}=0, i=0,1,\ldots,n\}$.
Then the sequence of partiotions $\{\Delta_{n}, n\geq 0\}$ have the following properties:
\begin{itemize}
	\item[(i)] $\Delta_{n}=\{A_{n}^{j}=T^{j}(A_{n}^{0}); j=0,1,\ldots, k\cdot (k+1)^{n-1}-1\}$ for $n\geq0$;
	\item [(ii)] $\Delta_{n+1}$ is a refinement of $\Delta_{n}$ and $A_{n}^{i}=A_{n+1}^{i}\cup A_{n+1}^{i+k\cdot (k+1)^{n-1}}\cup A_{n+1}^{i+2k\cdot (k+1)^{n-1}}\cup\cdots \cup A_{n+1}^{i+k^{2}(k+1)^{n-1}}$ for $i=0,1,\ldots, k\cdot(k+1)^{n-1}-1$ and $n\geq0$.
\end{itemize}
It then follows (see for example \cite{B}), that $S$ is conjugate to the product odometer defined on $(Z,\mu)$ and therefore is AT. This implies that the associated flow of $\mathcal{R}$ is AT and consequently, $N$ is an ITPFI factor.
\end{proof}

\subsection*{Acknowledgement. }This work was supported by a grant from Romanian Ministry of Research and Innovation, CNCS - UEFISCDI, project number  PN III-P4-ID-PCE-2016-0823 within PNCDI III.

\end{document}